\documentclass[a4paper,9pt]{article}
\usepackage{amsfonts}
\usepackage{amssymb,amsthm}
\usepackage{amsmath}

\newcommand{\R}{\mathbb{R}}

\newtheorem{theorem}{Theorem}
\newtheorem{lemma}[theorem]{Lemma}

\newtheorem{corollary}[theorem]{Corollary}

\def\per{\mathrm{Per}}

\def\H{{\mathcal H}}
\def\h{{\mathcal H^{m-1}}}

\def\I{{\mathcal I}}

\newcommand{\be}{\begin{equation}}
\newcommand{\ee}{\end{equation}}
\newcommand{\bee}{\begin{equation*}}
\newcommand{\eee}{\end{equation*}}
\newcommand{\bp}{\begin{proof}}
\newcommand{\ep}{\end{proof}}

\def\O{\Omega}
\def\o{\Omega_{m,k}}
\def\l{\lambda}
\def\ll{\lambda^*}

\begin{document}
\title{On the minimization of Dirichlet eigenvalues of the Laplace operator}
\author{{M. van den Berg\thanks{Research supported by The Leverhulme Trust, Research Fellowship 2008/0368}, M. Iversen} \\
School of Mathematics, University of Bristol\\
University Walk, Bristol BS8 1TW\\
United Kingdom\\
\texttt{M.vandenBerg@bris.ac.uk}\\
\texttt{Mette.Iversen@cantab.net}}
\date{Journal of Geometric Analysis {\bf 23}, 660--676 (2013)}\maketitle
\vskip 3truecm \indent
\begin{abstract}\noindent We study variational problems of the
form
$$\inf \{\lambda_k(\Omega): \Omega\ \textup{open in}\ \R^m,\ T(\Omega) \le 1 \},$$
where $\lambda_k(\Omega)$ is the $k$'th eigenvalue of the
Dirichlet Laplacian acting in $L^2(\Omega)$, and where $T$ is a
non-negative set function defined on the open sets in $\R^m$,
which is invariant under isometries, additive on disjoint families
of open sets, and is such that the ball with $T(B)=1$ is a
minimiser for $k=1$. Upper bounds are obtained for the number of
components of any bounded minimiser if $T$ satisfies a scaling
relation. For example we show that if $T$ is Lebesgue measure and
if $k\le m+1$ then any bounded minimiser has at most $7$
components. We also consider variational problems over open sets
$\Omega$ in $\R^m$ involving the $(m-1)$ - dimensional Hausdorff
measure of $\partial \Omega$.
\end{abstract}
\vskip 1truecm \noindent \ \ \ \ \ \ \ \  { Mathematics Subject
Classification (2000)}: 49Q10; 49R50; 35P15.
\begin{center} \textbf{Keywords}: Variational problems, Dirichlet eigenvalues.\\
\end{center}
\mbox{}\newpage
\section{Introduction\label{sec1}}
Let $\Omega$ be an open set in Euclidean space $\R^m \; (
m=2,3,\cdots)$, with boundary $\partial \Omega$, and let
$-\Delta_{\Omega}$ be the Dirichlet Laplacian acting in
$L^2(\Omega)$. It is well known that if $\Omega$ has finite
Lebesgue measure $|\Omega|=\int 1_{\Omega}$ then
$-\Delta_{\Omega}$ has compact resolvent, and the spectrum of
$-\Delta_{\Omega}$ is discrete and consists of eigenvalues
$\lambda_1(\Omega)\le\lambda_2(\Omega)\le \cdots$ with
$\lambda_j(\Omega)\rightarrow\infty$ as $j\rightarrow\infty$. The
Faber-Krahn inequality (Theorem 3.2.1 in \cite{H}) asserts that
\begin{equation}\label{e1}
\lambda_1(\Omega)\ge\lambda_1(B_m)\left(\frac{|B_m|}{|\Omega|}\right)^{2/m},
\end{equation}
where $B_m=\{x\in \R^m:|x|<1\}.$ By scaling we see that we have
equality in \eqref{e1} if $\Omega$ is any ball.

The Krahn-Szeg\"o inequality (Theorem 4.1.1 in \cite{H}) asserts
that
\begin{equation}\label{e2} \lambda_2(\Omega)\ge
2^{2/m}\lambda_1(B_m)\left(\frac{|B_m|}{|\Omega|}\right)^{2/m},
\end{equation}
where we have equality if $\Omega$ is the union of two disjoint
balls with equal measure. For higher Dirichlet eigenvalues ($k>2$) it
is not known whether the variational problem
\begin{equation}\label{e3}
\inf \{\lambda_k(\Omega) :\Omega\ \textup{open in}\ \R^m ,\ |\Omega|
\le 1 \}
\end{equation}
has a minimiser. However, it has been shown that if $k=3$, and if
the collection of open sets in \eqref{e3} is enlarged to the
quasi-open sets then a minimiser exists \cite{BH}. Open Problem 8
in \cite{H} asks to show that the minimiser for $k=3$ in
\eqref{e3} is a ball if $m=2,3$ or the union of three pairwise
disjoint balls with measure $1/3$ each if $m>3$. This suggests
that for large $k$ and large $m$ the number of components of a
minimiser of \eqref{e3} may be large. In Theorem \ref{the0} below
we obtain upper bounds for the number of components, denoted by
$\omega_{m,k}$, of any bounded minimiser $\Omega_{m,k}$ of
\eqref{e3}.

\begin{theorem}\label{the0}
\noindent If $\Omega_{m,k}$ is a bounded minimiser of \eqref{e3}
then
\begin{enumerate}
\item[i.]\hspace{16mm}$\omega_{m,k}\le k.$
\item[ii.]
\begin{equation*}\label{aa}
\omega_{m,k}\leq \begin{cases}1,\ m=2,3,\ k=3,\cdots,m+1, \\
2,\ m=4,\cdots,7,\ k=4,\cdots,m+1, \\
3,\ m=8, \cdots,19, \ k=5,\cdots,m+1,\\
4,\ m=20,\cdots,60, \ k=6,\cdots,m+1,\\
5,\ m=61,\cdots,548, \ k=7,\cdots,m+1,\\
6,\ m=549,\cdots, \ k=8,\cdots,m+1. \end{cases}
\end{equation*}
\end{enumerate}
\end{theorem}

We infer from $ii.$ that for $8\le k\le m+1$ the number of
components of a bounded minimiser of \eqref{e3} is at most $6$.
From $i.$ we have that for $k\le 7$ the number of components of a
bounded minimiser of \eqref{e3} is at most $7$. So for $k\le m+1$
a bounded minimiser of \eqref{e3} has at most $7$ components. We
recover the known fact that any bounded minimiser for the third
eigenvalue in $\R^2$ and $\R^3$ of \eqref{e3} is connected
\cite{WK}. Here we also obtain connectedness of any bounded
minimiser for the fourth eigenvalue in $\R^3$.

Below we state and prove a more general result of which Theorem
\ref{the0} is a special case.

\begin{theorem} \label{the}
Suppose $T$ is a non-negative set function defined on the open
sets in $\R^m$ which satisfies
\begin{enumerate}
\item[(a)] $T(\Omega)<\infty$ implies that the spectrum
of $-\Delta_{\Omega}$ is discrete.
\item[(b)]
$T(\cup_{\Omega \in \I} \Omega)=\sum_{\Omega \in \I} T(\Omega)$ if
$\I$ is a disjoint collection of open sets.
\item[(c)] There is $\beta>0$ such that for $\Omega$ open
in $\R^m$ and $\alpha>0$, $T(\alpha \Omega) = \alpha^{\beta}
T(\Omega)$.
\item[(d)] $\inf \{\lambda_1(\Omega) :\Omega\
\textup{open in}\ \R^m ,\ T(\Omega) \le 1 \}$ is minimised by the
ball $B \subset\R^m$ with $T(B)=1$.
\item[(e)] $T$ is invariant under isometries of $\R^m$.
\end{enumerate}
If $\Omega_{m,k}$ is a bounded minimiser of
\begin{equation}\label{a2}
\inf \{\lambda_k(\Omega) : \Omega\ \textup{open in}\ \R^m ,\
T(\Omega) \le 1 \},
\end{equation}
then
\begin{enumerate} \item[i.]\hspace{29mm}$\omega_{m,k}\le k$.
\item[ii.]For
$m=2,3,\cdots$ and $ k >
\lfloor(\lambda_k(B_m)/\lambda_1(B_m))^{\beta/2}\rfloor,$
\begin{equation}\label{a}
\omega_{m,k}\le
\lfloor(\lambda_k(B_m)/\lambda_1(B_m))^{\beta/2}\rfloor-1,
\end{equation}
where $\lfloor \cdot \rfloor$ denotes the integer part.
\end{enumerate}
\end{theorem}

It is easily seen that Lebesgue measure satisfies the hypotheses
of Theorem \ref{the} with $\beta=m$. Hence Theorem \ref{the}
implies Theorem \ref{the0}. An example of a set function with
$\beta=m+2$ is the torsional rigidity. In the Appendix in Section
\ref{sec5} we recall the definition of torsional rigidity and show
that it satisfies (a). It follows directly from its definition in
\eqref{a16} and \eqref{a17} below that the torsional rigidity
satisfies (b), (c) with $\beta=m+2$, and (e). In \cite{K1} and
\cite{K2} it was shown that (d) holds for the torsional rigidity
if $m=2$. The method of proof in these papers extends to all $m$
\cite{K1}. The bound on the number of components for any bounded
minimiser with a torsional rigidity constraint is given below.
\begin{corollary}\label{cor1}
If $T$ is a constraint which satisfies the hypotheses of Theorem
\ref{the} with $\beta=m+2$ and if $\Omega_{m,k}$ is a bounded
minimiser of \eqref{a2} then \begin{enumerate}
\item[i.]\hspace{16mm} $\omega_{m,k}\le k$.
\item[ii.]
\begin{equation}\label{ab}
\omega_{m,k}\leq \begin{cases}4,\ m=5,\cdots,26,\ k=6,\cdots,m+1, \\
5,\ m=27,\cdots,430,\ k=7,\cdots,m+1, \\
6,\ m=431,\cdots, \ k=8,\cdots,m+1. \end{cases}
\end{equation}
\end{enumerate}
\end{corollary}

We infer that, similarly to the lines below Theorem \ref{the0},
for $k\le m+1$ the number of components of a bounded minimiser of
\eqref{a2} with $\beta=m+2$ is at most $7$. Recall that
\begin{equation}\label{e41} \lambda_2(B_m)=\cdots =
\lambda_{m+1}(B_m)=j_{m/2}^2,
\end{equation}
and \begin{equation}\label{e41a} \lambda_1(B_m)=j_{(m-2)/2}^2,
 \end{equation}
where $j_{\nu}$ is the first positive zero of the Bessel function
$J_{\nu}$. Hence for $k\le m+1$ and $\beta=m+2$ we have that the
hypotheses on $k$ in part ii. of Theorem \ref{the} reads
\begin{equation}\label{e41b}
m+1 \ge k>
\lfloor(\lambda_k(B_m)/\lambda_1(B_m))^{\beta/2}\rfloor=\lfloor(j_{m/2}/j_{(m-2)/2})^{m+2}\rfloor.
\end{equation} The set of $k$ satisfying inequality \eqref{e41b} is non-empty if and only if $m\ge 5$.
So for $1<k\le m+1$ and $m=2,3,4$ Theorem \ref{the} gives only
that $\omega_{m,k}\le k$. This explains the absence of the cases
$m=2,3,4$ and $k\le m+1$ in \eqref{ab}.

The following variational problem was considered in \cite{BH2}.
\begin{equation}\label{e5}
\inf\{\lambda_2(\Omega): \Omega\ \textup{open and bounded in}\
\R^m,\per(\Omega)\le1 \},
\end{equation}
where the perimeter of a measurable set $\Omega$ is defined by
\begin{equation*}
\per(\Omega)=\int_{\R^m}|\nabla 1_{\Omega}|
\end{equation*}
in the sense of $BV$ functions, with $\per (\Omega)=+ \infty$ if
$1_\Omega $ is not a $BV$ function \cite{AFP}. There it was shown
that if $m=2$ then there exists a minimiser, which is convex, and
$C^\infty$. Moreover its boundary contains exactly two points
where the curvature vanishes.

It is easy to construct other minimisers of \eqref{e5}. Let
$\Omega_{m,2}$ be a minimiser of \eqref{e5}, and let $L$ be the
nodal set of a second Dirichlet eigenfunction for $\Omega_{m,2}$.
Then $\per (\Omega_{m,2} \setminus L)= \per(\Omega_{m,2})$ since
$|L|=0$. Since $\lambda_2(\Omega_{m,2})$ equals the first
eigenvalue of either of the nodal domains, we have that
$\lambda_2(\Omega_{m,2})=\lambda_2(\Omega_{m,2} \setminus L)$.
Hence $\Omega_{m,2} \setminus L$ is a minimiser of \eqref{e5}
which is not connected. If $C$ is any closed subset of $L$ then
$\Omega_{m,2} \setminus C$ is also a minimiser. In order to be
able to study topological properties such as connectedness we
replace $\per(\Omega)$ in \eqref{e5} by the $(m-1)$ - dimensional
Hausdorff measure of $\partial\Omega$ denoted by
$\h(\partial\Omega)$, and consider the following variational
problem instead.

\begin{equation}\label{e7}
\inf\{\lambda_k(\Omega): \Omega\ \textup{open in}\ \R^m ,|\Omega|<
\infty, \H^{m-1}(\partial \Omega)\le1 \}.
\end{equation}

We note that \eqref{e7} is not of the form \eqref{a2}. An
additional constraint $|\Omega|< \infty$ has to be inserted to
guarantee discreteness of the Dirichlet spectrum. Without this
constraint the complement of the closed ball $\overline{B}$ with
$\h(\partial\overline{B})=1$ is an open set with $(m-1)$ -
dimensional Hausdorff measure of its boundary equal to $1$ and
Dirichlet spectrum equal to $[0,\infty)$. We also note that
Hausdorff measure does not satisfy (b) in Theorem \ref{the} as it
is only subadditive. However, Hausdorff measure of the boundary is
supported on all of the topological boundary, whereas the
perimeter is supported on the reduced boundary \cite{BB}.

Throughout the paper we denote for a set $E\subset \R^m$ its
interior by $\textup{int}(E)$, its closure by $\overline{E}$, and
$E^*=\textup{int}(\overline{E})$. For $x\in \R^m, R>0$ we let
$B(x;R)=x+RB_m$. We denote the infima in \eqref{a2} and in
\eqref{e7} by $\ll_k$. Our main results for \eqref{e7} are the
following.

\begin{theorem}\label{The1}
\noindent
\begin{enumerate}
\item[i.] If $m=2$, and $k=2,3,\cdots$ then \eqref{e7} has a minimiser which is open,
bounded and convex.
\item[ii.] Let $\Omega_{m,k}$ be a minimiser of \eqref{e7}. (a) If K is a relatively closed subset
of the nodal set $L$ of the $k$'th Dirichlet eigenfunction for
$\Omega_{m,k}$ with $\h(K)=0$ then $\Omega_{m,k}\setminus K$ is
also a minimiser of \eqref{e7}. (b) $\Omega_{2,k}$ is connected
for all $k=1,2,\cdots$.
\item[iii.]
If $m\rightarrow \infty$ then
\begin{equation}\label{e9}
\ll_2 =\lambda_1(B_m)(\h(\partial B_m))^{2/(m-1)}(1+(\log 4)m^{-1}
+ O(m^{-2})).
\end{equation}
\item[iv.] If $m=2,3,\cdots$ then $\Omega_{m,2}$ is not a ball.
\end{enumerate}
\end{theorem}

In Theorem \ref{The2} below we give some topological properties of
minimisers of \eqref{e7}.

\begin{theorem}\label{The2}
If $\Omega_{m,k}$ is a minimiser of \eqref{e7} then
\begin{enumerate}
\item[i.]$\Omega_{m,k}^*$ is a minimiser of
\eqref{e7}.
\item[ii.]$\R^m \setminus \Omega_{m,k}^*$ is connected.
\item[iii.]$\Omega_{m,2}$ is connected $(\omega_{m,2}=1)$ for
$m=3,4,\cdots$.

If $\Omega_{m,k}$ is a bounded minimiser, and if $k=3,4,\cdots $,
and $ m=3,4,\cdots $, then
\begin{equation}\label{e10}
\omega_{m,k}\le
\min\{\lfloor(k+1)/2\rfloor,1+\lfloor2^{-(m-1)/m}((\lambda_k(B_m)/\lambda_1(B_m))^{(m-1)/2}-1)\rfloor\}.
\end{equation}
In particular $\omega_{m,k}\le \lfloor(k+1)/2\rfloor$, and
\begin{displaymath}
\omega_{m,k}\le \left\{ \begin{array}{ll}
1, & \textrm{$m=3,4,5,\ \ k=3,\cdots,m+1,$}\\
2, & \textrm{$m=6,\cdots,24,\ \  k=5,\cdots,m+1,$}\\
3, & \textrm{$m=25,\cdots,587,\ \  k=7,\cdots,m+1,$}\\
4, & \textrm{$m=588,\cdots,\ \  k=9,\cdots,m+1.$}
\end{array} \right.
\end{displaymath}
\end{enumerate}
\end{theorem}

At present we do not know whether there exists a minimiser of
\eqref{a2} or of \eqref{e7} with $m>2$, and $k=2,3,\cdots$, and if
so whether such a minimiser has a smooth boundary. The proofs in
this paper do not rely on any such smoothness properties.

A key ingredient in the proof of Theorem \ref{The2} is the
isoperimetric inequality. Recall (Theorem 3.46 in \cite{AFP}, \cite{C}) that
for a measurable set $\Omega \subset \R^m$ with $|\Omega| <
\infty$,
\begin{equation*}
|\Omega|\le |B_m|\left(\frac{\per (\Omega)}{\per
(B_m)}\right)^{m/(m-1)}.
\end{equation*}
This combined with $\per(\Omega)\le \h(\partial \Omega)$ and
$\per(B_m)= \h(\partial B_m)$ gives the isoperimetric inequality
for the $(m-1)$ - dimensional Hausdorff measure
\begin{equation}\label{perhaus}
|\Omega|\le |B_m|\left(\frac{\h(\partial \O)}{\h(\partial
B_m)}\right)^{m/(m-1)}.
\end{equation}
Inequality \eqref{perhaus} is well known. See for example
\cite{A}, where it was stated for bounded regions in $\R^m$. By
Faber-Krahn \eqref{e1} and \eqref{perhaus} we obtain the
isoperimetric inequality
\begin{equation}\label{e11}
\lambda_1(\Omega)\ge \lambda_1(B_m)\left(\frac{\h(\partial
B_m)}{\h(\partial \Omega)}\right)^{2/(m-1)}.
\end{equation}
By Krahn-Szeg\"o \eqref{e2} and \eqref{perhaus} we have that
\begin{equation}\label{e11a}
\lambda_2(\Omega)\ge 2^{2/m}\lambda_1(B_m)\left(\frac{\h(\partial
B_m)}{\h(\partial \Omega)}\right)^{2/(m-1)}.
\end{equation}
Inequality \eqref{e11a} is not isoperimetric since \eqref{e2} and
\eqref{perhaus} are isoperimetric for non-isometric sets.

This paper is organized as follows. In Section \ref{sec2} we prove
Theorem \ref{the}. The proofs of Theorems \ref{The1} and
\ref{The2} are deferred to Sections \ref{sec3} and \ref{sec4}
respectively.

\section{Proof of Theorem \ref{the}}\label{sec2}

Throughout the paper we say that a component $G$ of a minimiser
$\Omega_{m,k}$ of \eqref{a2} or of \eqref{e7} supports $l$
eigenvalues of $\o$ if $\#\{\lambda_i(G)\le \l_k(\o)\}=l$.

\begin{lemma} \label{lambda2} Suppose $T$ satisfies the hypotheses of Theorem \ref{the}.
If $\Omega$ is an open set in $\R^m$ with $T(\Omega)<\infty $ then
\begin{equation}\label{l1}
\lambda_1(\Omega)\ge
\lambda_1(B_m)\left(\frac{T(B_m)}{T(\Omega)}\right)^{2/\beta},
\end{equation}
and
\begin{equation}\label{l2}
\lambda_2(\Omega)\ge
2^{2/\beta}\lambda_1(B_m)\left(\frac{T(B_m)}{T(\Omega)}\right)^{2/\beta}.
\end{equation}
\end{lemma}
\begin{proof} The proof of \eqref{l1} follows directly from hypotheses (b) and
(c) in Theorem \ref{the}. To prove \eqref{l2} we let $\phi_2$ be
the second eigenfunction of the Dirichlet Laplacian on $\Omega$,
and let $\Omega^+=\{x \in \Omega:\phi(x)>0\}$ and $\Omega^-=\{x
\in \Omega:\phi(x)<0\}$. Then
$\lambda_2(\Omega)=\lambda_1(\Omega^+)=\lambda_1(\Omega^-).$ By
\eqref{l1} applied to both $\Omega^+$ and $\Omega^-$ respectively
we obtain that
\begin{equation*}
\lambda_2(\Omega)\ge \lambda_1(B_m)T(B_m)^{2/{\beta}}\max
\{T(\Omega^+)^{-2/{\beta}}, T(\Omega^-)^{-2/{\beta}}\},
\end{equation*}
and \eqref{l2} follows since $T(\Omega^+)+T(\Omega^-)= T(\Omega)$.
\end{proof}
Note that equality in \eqref{l2} implies that $\Omega$ is the
union of two disjoint balls with equal measure. This extends the
Krahn-Szeg\"o inequality to the class of set functions satisfying
the hypotheses of Theorem \ref{the}.

\begin{lemma}\label{constraint} \begin{enumerate}
\item[(i)]If $G$ is an open set with $\l_k(G)\le \l_k^*$, where $\l_k^*$ is
as in \eqref{a2} or \eqref{e7} then $T(G)\ge 1$ or $\h(\partial
G)\ge 1$ respectively.
\item[(ii)]If $\Omega_{m,k}$ is a minimiser either of \eqref{a2} or of
\eqref{e7} then $T(\o)=1$ or $\h(\partial\o)=1$
respectively.\end{enumerate}
\end{lemma}
\begin{proof} (i) Suppose $G$ is an open set with $\l_k(G)\le
\l_k^*$ and $T(G)<1$. Let $\alpha>0$ be such that $T(\alpha G)=1$.
By the hypothesis (c) of Theorem \ref{the}, $\alpha>1$. Then
$\l_k(\alpha G)= \alpha^{-2}\l_k(G)\le \alpha^{-2}\l_k^*< \l_k^*$
contradicting the definition of $\l_k^*$ in \eqref{a2}. (ii) Since
$\o$ is a minimiser of \eqref{a2} $T(\o)\le 1$. By (i) $T(\o)\ge
1$. Hence $T(\o)=1$. The proofs of the assertions for $\h(\partial
G)$ and $\h(\partial\o)$ are similar.
\end{proof}

\begin{lemma}\label{Lem3}
If $\Omega_{m,k}$ is a minimiser of \eqref{a2} or of \eqref{e7}
then $\Omega_{m,k}$ has at most $k$ components, i.e. $\omega_{m,k}
\le k$.
\end{lemma}
\begin{proof}
First suppose that $\o$ is a minimiser of \eqref{a2}. Since $\o$
is open we have that
\begin{equation}\label{e22}
\Omega_{m,k}= \cup_{i \in I}G_i,
\end{equation}
where the $G_i, i \in I$ are pairwise disjoint, open, non-empty,
and connected, and $I$ is either finite or countably infinite. We
relabel the $G_i$'s such that $\l_1(G_1)\le\l_1(G_2)\le \cdots $.
Let $l=\min\{k,\max\{j:\l_1(G_j)\le\l_k(\o)\}\}$. So $l\le k$. Let
$G=G_1\cup \cdots \cup G_l$. Then $G$ is open and $\l_k(G)\le
\ll_k$. If $\#I\ge k+1$ then $\o \setminus G$ is non-empty and
open. By additivity of $T$, $T(G)<T(\o)=1$ which is impossible by
Lemma \ref{constraint} (i). Hence $\omega_{m,k}=\#I\le k$.

Next suppose that $\o$ is a minimiser of \eqref{e7}. By the
argument above it suffices to show that if $\o$ is as in
\eqref{e22} then
\begin{equation}\label{e22a} \h(\partial \o) = \sum_{i\in I} \h(\partial
G_i).
\end{equation} If not then there exists $ i, j \in I$, $i \ne j $,
such that  $\H^{m-1}((\partial G_i) \cap (\partial G_j)) > 0$. By
Theorem \ref{The2}(i) and Lemma \ref{constraint}(ii)
$\Omega_{m,k}^*$ is then a minimiser with $\h(\partial
\Omega_{m,k}^*)\le 1-\H^{m-1}(\partial G_i \cap
\partial G_j)<1$. The latter is impossible by Lemma
\ref{constraint}(ii).
\end{proof}
The above shows in fact that any non-overlapping rearrangement of
the components of a minimiser of \eqref{e7} satisfies
\eqref{e22a}.

\begin{lemma}\label{Lem5}
Let $G$ be a component of a bounded minimiser of \eqref{a2} with
$T(G)=c$ or of \eqref{e7} with $\h(\partial G)=c$ respectively.
Denote the eigenvalues of $-\Delta_{G}$ which are not larger than
$\lambda_k^*$ by $\lambda_1(G), \cdots, \lambda_j(G)$. Then
$\lambda_j(G)=\lambda_k^*$, and $G$ is a minimiser of
\begin{equation}\label{e22b}\inf \{\lambda_j(\Omega) : \Omega\ \textup{open in}\ \R^m
,\ T(\Omega) = c \},\end{equation} or of
\begin{equation*}
\inf\{\lambda_j(\Omega): \Omega\ \textup{open in}\ \R^m ,|\Omega|<
\infty, \H^{m-1}(\partial \Omega)=c \}\end{equation*}
respectively.
\end{lemma}
\begin{proof}
Let $G$ be a component of a bounded minimiser $\o$ of \eqref{a2}
with $T(G)=c$. Suppose $\l_j(G)<\l_k(\o)$. Let $\alpha<1$ be such
that $\l_j(\alpha G)=\lambda_k^*$. Since $\o$ is bounded all its
components are bounded, and we may rearrange these, if necessary,
such that $(\o \setminus G)\cap ( \alpha G)=\emptyset.$ Then
$\l_k((\o \setminus G)\cup (\alpha G))\le \l_k(\o)$, and $T((\o
\setminus G)\cup (\alpha G))=1-c+\alpha^{\beta}c<1$. The latter is
impossible by Lemma \ref{constraint}(i).

Next suppose that $G$ is not a minimiser of \eqref{e22b}. If $A$
is a minimiser of \eqref{e22b} then $\l_j(A)<\l_j(G)=\l_k^*$. Let
$\alpha<1$ be such that $\l_j(\alpha A)=\l_k^*$. Rearrange if
necessary the components of $\o$ such that $(\o \setminus
G)\cap(\alpha A)=\emptyset$. Then $\l_k((\o \setminus G)\cup (
\alpha A))\le \ll_k$, and $T((\o \setminus G)\cup ( \alpha
A))=1-c+\alpha^{\beta}c<1$. The latter is impossible by Lemma
\ref{constraint}(i).

The proof of the corresponding assertion for components of bounded
minimisers of \eqref{e7} is similar.
\end{proof}

\begin{proof}[Proof of Theorem \ref{the}.]
If $\Omega_{m,k}$ is a minimiser of \eqref{a2} then it is of the
form
\begin{equation*}
\Omega_{m,k}=\cup_{i=1}^{\omega_{m,k}} G_i,
\end{equation*}
where the $G_i$'s are as in the proof of Lemma \ref{Lem3}. We
denote the eigenvalues of $G_i$ which are not strictly larger than
$\l_k^*$ by $\lambda_1(G_i),\cdots, \lambda_j(G_i)$, and put
$c_i=T(G_i)$. By Lemma \ref{Lem5},
\begin{equation}\label{a5}\lambda_j(G_i)=\l_k^*,
\end{equation}
and $G_i$ is a minimiser of \eqref{e22b} with $c=c_i$.

Let $\omega_{m,k}=k_1+k_2$, where $G_1,\cdots, G_{k_1}$ support
one eigenvalue each, and each of $G_{k_1+1},\cdots,G_{k_1+k_2}$
supports at least two eigenvalues. If $\omega_{m,k}=k,$ then
$\Omega_{m,k}$ is the union of $k$ pairwise disjoint balls with
equal measure, and $\l_k^* = \lambda_1(B)k^{2/\beta}$. Combining
this with
\begin{equation}\label{a13}
\l_k^*\le \lambda_k(B),
\end{equation}
gives
\begin{equation*}
k \le
(\lambda_k(B)/\lambda_1(B))^{\beta/2}=(\lambda_k(B_m)/\lambda_1(B_m))^{\beta/2}.
\end{equation*}
Hence if $k > (\lambda_k(B_m)/\lambda_1(B_m))^{\beta/2}$ then
$k_2\ge 1.$

By hypothesis (d) and \eqref{a5} each of the components
$G_1,\cdots,G_{k_1}$ is a ball with $T(G_1)=\cdots=T(G_{k_1})=:a$.
So
\begin{equation}\label{a8}
\l_k^*=\lambda_1(G_1)=\cdots =
\lambda_1(G_{k_1})=\lambda_1(B)a^{-2/\beta}.
\end{equation}

Let $G_i$ be one of the remaining $k_2$ components supporting at
least two eigenvalues. By Lemma \ref{lambda2}
\begin{equation}\label{a9}
\l_k^*=\lambda_j(G_i)\ge \lambda_2(G_i)\ge
2^{2/\beta}\lambda_1(B)T(G_i)^{-2/\beta}.
\end{equation}
But
\begin{equation}\label{a10}
\min_{i\in \{k_1+1,\cdots,k_1+k_2\}}T(G_i)\le
k_2^{-1}\sum_{i=k_1+1}^{k_1+k_2}T(G_i)=k_2^{-1}(1-k_1a).
\end{equation}
Combining \eqref{a8}, \eqref{a9} and \eqref{a10} we obtain that
\begin{equation}\label{a11}
\l_k^*\ge \lambda_1(B)\max
\left\{a^{-2/\beta},(2k_2(1-k_1a)^{-1})^{2/\beta}\right\}.
\end{equation}
The right hand side of \eqref{a11} attains its minimum for
$a=(k_1+2k_2)^{-1}$, and so by \eqref{a11}
\begin{equation}\label{a12}
\l_k^* \ge \lambda_1(B)(k_1+2k_2)^{2/\beta}\ge
\lambda_1(B)(\omega_{m,k}+1)^{2/\beta}.
\end{equation}
Combining \eqref{a12} with \eqref{a13} and Lemma \ref{Lem3}
implies \eqref{a}.
\end{proof}

Since the minimiser of \eqref{a2} for $k=2$ is the union of two
disjoint balls with equal measure it follows that each of the
$G_i$'s support either one eigenvalue or at least three
eigenvalues. Thus $k\ge k_1+3k_2.$ This can give additional
information. Consider for example any bounded minimiser of
\eqref{a2} with  $k=4$ or $k=5$, $ m=4,\cdots,7$ and $T$ Lebesgue
measure. By Theorem \ref{the} it has at most two components, and
as no component supports two eigenvalues the minimiser is either
connected or is the union of a ball supporting one eigenvalue with
a component supporting three (if $k=4$) or four (if $k=5$)
eigenvalues respectively.

\section{Proof of Theorem \ref{The1}}\label{sec3}

\begin{proof}[Proof of Theorem \ref{The1}](i) Let $m=2$, and let
$(\Omega_n)$ be a minimising sequence of \eqref{e7}. By Lemma
\ref{Lem3} we have that $\Omega_n= \cup_{i=1}^k A_{n,i}$, where
the $A_{n,i}, i=1,\cdots,k$ are pairwise disjoint, open and
connected. By translational and rotational invariance we may
rearrange the $A_{n,i}$'s such that they remain disjoint but such
that $\overline{\cup_{i=1}^kA_{n,i}}$ is connected. Taking the
convex envelope of $\overline{\cup_{i=1}^kA_{n,i}}$ does not
increase $\H^1(\partial (\overline{\cup_{i=1}^kA_{n,i}})$ nor does
$\lambda_k(\textup{int}(\overline{\cup_{i=1}^kA_{n,i}}))$
increase. We denote the resulting sequence of convex sets again by
$(\Omega_n)$. It is clear that the diameter of $\Omega_n$ is
bounded by $1/2$. By translating the $\Omega_n$'s we may assume
that they are contained in the closed ball with radius $1$ in
$\R^2$. Following the proof of Theorem 2.1 in \cite {BH}, there
exists a subsequence of $(\Omega_n)$ again denoted by $(\Omega_n)$
which converges to a convex set $\Omega$ in the Hausdorff metric.
Then $\H^1(\partial\Omega)= \per(\Omega)$ by the convexity of
$\Omega$.
 By the lower semicontinuity for the
perimeter (Proposition  2.3.6 in \cite{HP}) we have that
$\H^1(\partial \Omega)\le 1$. Finally $\lambda_k(\Omega_n)
\rightarrow \lambda_k(\Omega)$ by Proposition 2.4.6 in \cite{BB}. We
may choose $\Omega$ open. Its diameter is bounded by $1/2$.

(ii)(a) Since $\h(K)=0$ we have that
$\h(\partial(\Omega_{m,k}\setminus K))=\h(\partial\Omega_{m,k})$.
Since $K$ is a subset of the nodal set for a $k$'th eigenfunction
for $\Omega_{m,k}$ we have that $\lambda_k(\Omega_{m,k}\setminus
K)=\lambda_k(\Omega_{m,k})$, and $\Omega_{m,k}\setminus K$ is a
minimiser too. Note that it follows by the proof under (i) that
all minimisers of \eqref{e7} for $m=2$ are convex up to a set of
capacity $0$ or up to a subset of the nodal line with one
dimensional Hausdorff measure $0$.

(b) Let $\Omega_{2,k}$ be a minimiser of \eqref{e7} for $m=2$, and
let $\tilde{\Omega}_{2,k}$ be its open convex envelope. Then
$\tilde{\Omega}_{2,k}$ is open and connected. If
$K=\tilde{\Omega}_{2,k}\setminus \Omega_{2,k}$ then $\H^1(K)=0$,
and $K$ does not partition $\tilde{\Omega}_{2,k}$. Hence
$\Omega_{2,k}$ is connected.

(iii) To obtain a lower bound for $\ll_2$ we have by definition of
$\ll_2$ and \eqref{e11a} that
\begin{align}\label{e20}
\ll_2 &= \inf \{\lambda_2(\Omega)(\h(\partial
\Omega))^{2/(m-1)}:\Omega \ \text{open in } \R^m, |\Omega|<
\infty  \}\\ \nonumber &\ge 2^{2/m}\lambda_1(B_m)(\h(\partial
B_m))^{2/(m-1)}.
\end{align}
To obtain an upper bound for $\ll_2$ we choose for $\Omega$ the
union of two disjoint open balls each with boundary measure $1/2$.
This gives
\begin{equation}\label{e21}
\ll_2\le2^{2/(m-1)}\lambda_1(B_m)(\h(\partial B_m))^{2/(m-1)},
\end{equation}
and \eqref{e9} follows by (\ref{e20}) and \eqref{e21}.

(iv) Suppose that $k=2$ and that $B_m$ is a minimiser of
\eqref{e7}. Then $\lambda_2^*=\lambda_2(B_m)(\h(\partial
B_m))^{2/(m-1)}$. Then by \eqref{e21} we have that
\begin{equation}\label{e21a}
\lambda_2(B_m)\le 2^{2/(m-1)}\lambda_1(B_m).
\end{equation}
Hence \eqref{e21a} implies by \eqref{e41} and \eqref{e41a} that
\begin{equation}\label{e21b}
j_{m/2}\le 2^{1/(m-1)}j_{(m-2)/2}.
\end{equation}
However, \eqref{e21b} contradicts the numerical values of
$j_{(m-2)/2}$ and of $j_{m/2}$ for $3\le m <2^{15}$ of \cite{M}.
For $m\ge 2^{15}$ \eqref{e21b} contradicts the lower bound for
$j_{m/2}$ and the upper bound for $j_{(m-2)/2}$ as obtained from
\eqref{e43} below. Hence $\Omega_{m,2}$ is not a ball for
$m=3,4,\cdots$.

To show that $B_2$ is not a minimiser for \eqref{e7} with $k=m=2$
we consider the ellipse
\begin{equation*}
\Omega_t=\{(x_1,x_2) \in \R^2: x_1^2+(1+t)^{-2}x_2^2<1\},\ t>0.
\end{equation*}
An elementary calculation shows that for $t\rightarrow 0$
\begin{equation}\label{e21c}
\H^1(\partial
\Omega_t)=4\int_0^1dx(1-x^2)^{-1/2}(1+2tx^2+t^2x^2)^{1/2}=2\pi(1+t/2)+o(t).
\end{equation}
Let $\phi_t$ denote the Dirichlet eigenfunction corresponding to
$\lambda_2(\Omega_t)$. The nodal line of $\phi_t$ is the set
$\Omega_t\cap\{x_2=0\}$. Denote
$\Omega_{t,+}=\Omega_t\cap\{x_2>0\}$. Then
$\lambda_2(\Omega_t)=\lambda_1(\Omega_{t,+})$. Define for
$(x_1,x_2)\in \Omega_{t,+}$
\begin{equation*}
\psi_t(x_1,x_2)=\phi_0(x_1,(1+t)^{-1}x_2),
\end{equation*}
where $\phi_0=\lim_{t \rightarrow 0^+}\phi_t$, and restricted to
$\Omega_{0,+}$, is the first Dirichlet eigenfunction corresponding
to $\lambda_1(\Omega_{0,+})$. Then
\begin{equation}\label{e21e}
\int_{\Omega_{t,+}}\psi_t^2=(1+t)\int_{\Omega_{0,+}}\phi_0^2,
\end{equation}
and
\begin{equation}\label{e21f}
\int_{\Omega_{t,+}}|\nabla
\psi_t|^2=(1+t)\int_{\Omega_{0,+}}\left( \left(\frac{\partial
\phi_0}{\partial x_1}\right)^2+(1+t)^{-1}\left(\frac{\partial
\phi_0}{\partial x_2}\right)^2\right).
\end{equation}
Since $$\lambda_1(\Omega_{t,+})\le
\frac{\int_{\Omega_{t,+}}|\nabla
\psi_t|^2}{\int_{\Omega_{t,+}}\psi_t^2}\ ,$$ we have by
\eqref{e21e} and \eqref{e21f} that for $t\rightarrow 0$
\begin{align}\label{e21g}
\lambda_2(\Omega_t)&\le
\lambda_2(\Omega_0)-2t\frac{\int_{\Omega_{0,+}}\left(\frac{\partial
\phi_0}{\partial
x_2}\right)^2}{\int_{\Omega_{0,+}}\phi_0^2}+o(t)\nonumber
\\ &=\lambda_2(\Omega_0)\left(1-2t\int_{\Omega_{0,+}}\left(\frac{\partial
\phi_0}{\partial x_2}\right)^2\left(\int_{\Omega_{0,+}}|\nabla
\phi_0|^2\right)^{-1}\right)+o(t).
\end{align}
Since $\phi_0$ is given in polar coordinates by
\begin{equation}\label{e21h}
\phi_0(r,\theta)=J_1(j_1r)\sin \theta,\ 0<\theta < \pi,\ 0<r<1,
\end{equation}
we use \eqref{e21h}, $\int_0^{\pi}(\cos
\theta)^4d\theta=\int_0^{\pi}(\sin
\theta)^4d\theta=3\int_0^{\pi}(\cos \theta)^2(\sin
\theta)^2d\theta$, and \\$\int_0^1J_1'(j_1r)J_1(j_1r)dr=0$ to
verify that
\begin{equation}\label{e21i}
\int_{\Omega_{0,+}}\left(\frac{\partial \phi_0}{\partial
x_2}\right)^2=\frac{3}{4}\int_{\Omega_{0,+}}|\nabla \phi_0|^2.
\end{equation}
Combining \eqref{e21c}, \eqref{e21g}, and \eqref{e21i} we conclude
that for $t\rightarrow 0$
\begin{equation*}
(\H^1(\partial \Omega_t))^2\lambda_2(\Omega_t)\le(\H^1(\partial
\Omega_0))^2\lambda_2(\Omega_0)(1-t/2)+o(t)< \lambda_2^*.
\end{equation*}
Hence $\Omega_0=B_2$ is not a minimiser.
\end{proof}

\section{Proof of Theorem \ref{The2}}\label{sec4}

To prove Theorem \ref{The2}(i) suppose that $\Omega_{m,k}$ is a
minimiser of \eqref{e7}. Then $\Omega_{m,k}^*$ is open and
$\partial\Omega_{m,k}^* = \overline{\Omega_{m,k}^*}\setminus
\Omega_{m,k}^* \subset \overline{\Omega_{m,k}}\setminus
\textup{int}(\Omega_{m,k})=
\partial\Omega_{m,k},$ and hence
 $\h(\partial\Omega_{m,k}^*)\le 1.$
Also note that $\Omega_{m,k}^* \setminus \Omega_{m,k} \subset
\partial\Omega_{m,k} $ and so $|\Omega_{m,k}^* \setminus \Omega_{m,k}|=0$. Thus $|\Omega_{m,k}^*|\le |\Omega_{m,k}|<
\infty$. Finally $\Omega_{m,k} \subset \Omega_{m,k}^*$, which
implies $\lambda_k(\Omega_{m,k}^*)\le \lambda_k(\Omega_{m,k})$.
Therefore $\Omega_{m,k}^*$ is a minimiser of \eqref{e7}.

To prove Theorem \ref{The2}(ii) we note that $\R^m\setminus
\Omega_{m,k}^*$ is closed and hence its components are closed.
Suppose that $C$ is a component of $\R^m \setminus \Omega_{m,k}^*$
with $\h(\partial C)>0$ and $|C| < \infty$. This gives
$|\Omega_{m,k}^* \cup C| \le |\Omega_{m,k}^*|+|C| < \infty$ (a).
By monotonicity of Dirichlet eigenvalues $\lambda_k(
\Omega_{m,k}^* \cup C)\le \lambda_k( \Omega_{m,k}^*)$ (b). Also
$\partial C \subset
\partial \Omega_{m,k}^*$, and hence $\h(\partial (\Omega_{m,k}^* \cup C))= \h(
\partial \Omega_{m,k}^*)- \h(\partial C) < \h( \partial \Omega_{m,k}^*)\ $\ (c).
To show that $\Omega_{m,k}^* \cup C$ is open it suffices to show
that any $x \in\partial C$ is an interior point. Suppose to the
contrary that for all $\epsilon >0$, $B(x;\epsilon)\setminus
(\Omega_{m,k}^* \cup C)\ne \emptyset$. Then $x$ is a limit point
of another closed component of $\R^m \setminus \Omega_{m,k}^*$,
and so belongs both to that component and $C$. This contradicts
the maximality of $C$.  Hence $\Omega_{m,k}^* \cup C$ is open (d).
Then (a)-(d) contradict that $\Omega_{m,k}^*$ is a minimiser of
\eqref{e7}. Finally suppose $C$ is a component of $\R^m \setminus
\Omega_{m,k}^*$ with $\h(\partial C)=0$. Then  as above $C \subset
\textup{int}(\Omega_{m,k}^* \cup C)$, which combined with $C
=\partial C \subset \partial \Omega_{m,k}^*$ implies the
contradiction $C \subset \Omega_{m,k}^*.$ We conclude that all
components of $R^m \setminus \Omega_{m,k}^*$ have infinite
Lebesgue measure. Since $\h(\partial \Omega^*)\le 1$,
$\Omega_{m,k}^*$ cannot separate infinite components, and so $\R^m
\setminus \Omega_{m,k}^*$ is connected.

\begin{lemma}\label{Lem6} Let $B(\epsilon )=B(0;R)\cap\{x:x_1<R-\epsilon
\}$, and let
\begin{equation}\label{e25}
\Omega(\epsilon)= \cup_{j=0}^1\ B(2(R-\epsilon)je_1;R),
\end{equation}
where $e_1=(1,0,\cdots,0)$. Then
\begin{equation}\label{e25a}\lambda_2(\Omega(\epsilon))\le
\lambda_1(B(\epsilon))\le
\lambda_1(B(0;R))+O(\epsilon^{(m+1)/2}),\ \epsilon \rightarrow 0.
\end{equation}
\end{lemma}
\begin{proof} The first inequality in \eqref{e25a} follows by Dirichlet bracketing if we impose
Dirichlet boundary conditions on
$\Omega(\epsilon)\cap\{x_1=R-\epsilon\}$. To prove the second
inequality in \eqref{e25a} we denote the first Dirichlet
eigenfunction on $B(0;R)$ by $\phi$, and let $\chi$ be a
$C^\infty$ function on $\R^m$ depending on $x_1$ only, which is
decreasing in $x_1$ on $[R-2\epsilon,R-\epsilon]$, with
$|\nabla\chi(x)|\le2/\epsilon$, $\chi(x)=-1$ for $x_1\ge
R-\epsilon$, and $\chi(x)=0$ for $x_1\le R-2\epsilon$. Let
$\psi=(1+\chi)\phi$. We will use the variational principle with
test function $\psi$ to obtain an upper bound on
$\lambda_1(B(\epsilon))$. Recall that since $\partial B(0;R)$ is
smooth there exists $C$ depending on $m$ and on $R$ only such that
$\phi(x)\le C(R-|x|)$, and $|\nabla \phi(x)|\le C$. Firstly
\begin{align}\label{e24a}
&\int_{B(\epsilon)} |\nabla \psi|^2 =\int_{B(\epsilon)}\left(
|\nabla \phi|^2(1+\chi)^2+\phi^2|\nabla\chi|^2
+2\phi(1+\chi)\nabla\phi.\nabla\chi\right)\\
\nonumber &\le\int_{B(\epsilon)} |\nabla \phi|^2
+C^2\int_{B(\epsilon)-B(2\epsilon)}\left((R-|x|)^2|\nabla\chi|^2
+2C^2(R-|x|)|\nabla\chi|\right)\\
\nonumber &\le \int_{B(0)}|\nabla\phi|^2+24C^2|B(0)-B(2\epsilon)|.
\end{align}
Secondly
\begin{align}\label{e24b}
\int_{B(\epsilon)}\phi^2(1+\chi)^2&=\int_{B(0)} \phi^2(1+\chi)^2
\ge\int_{B(0)} \left(\phi^2+2\phi^2\chi\right)\\ \nonumber
&\ge\int_{B(0)}\left(\phi^2+2C^2\chi\right)
\ge\int_{B(0)}\phi^2-2C^2|B(0)-B(2\epsilon)|.
\end{align}
We conclude by \eqref{e24a} and \eqref{e24b} that for $\epsilon
\rightarrow 0$
\begin{equation*}
\lambda_1(B(\epsilon))\le\lambda_1(B(0;R))+O(|B(0;R)-B(2\epsilon)|)=\lambda_1(B(0;R))+O(\epsilon^{(m+1)/2}).
\end{equation*}
\end{proof}

\begin{lemma}\label{Lem7} Let $m=3,4, \cdots$, and let $k=2,3,4, \cdots$. If $\Omega_{m,k}$ is
a minimiser of \eqref{e7} then $\Omega_{m,k}$ has at most one
component supporting only one eigenvalue.
\end{lemma}
\begin{proof}
Suppose $\Omega_{m,k}$ has at least $2$ components say $G_1$ and
$G_2$ supporting only one eigenvalue each. By Lemma \ref {Lem5}
each of these components is a minimiser for the first eigenvalue,
and $\lambda_1(G_1)=\lambda_1(G_2)=\lambda_k^*$. Hence by
\eqref{e11} these components are balls with equal radius say $R$.
Let $\Omega (\epsilon )$ be as in \eqref{e25}.  An elementary
calculation shows that for $\epsilon \rightarrow 0 $
\begin{equation*}
\h (\partial \Omega (\epsilon ))= \h (\partial \Omega (0
))-2\Gamma ((m+1)/2)^{-1}(2\pi R \epsilon )^{(m-1)/2}(1+o(1)).
\end{equation*}
Let $L(\epsilon)>0$ be such that
\begin{equation*}
\h(\partial(L(\epsilon)\Omega(\epsilon)))=\h(\partial(\Omega(0))).
\end{equation*}
Then
\begin{equation}\label{e26b}
L(\epsilon)=1+C\epsilon^{(m-1)/2}(1+o(1)),
\end{equation} as
$\epsilon \rightarrow0$ for some $C>0$ depending on $m$ and on $R$
only. By scaling, Lemma \ref{Lem6} and \eqref{e26b}
\begin{equation*}
\lambda_2(L(\epsilon)\Omega (\epsilon))=
L(\epsilon)^{-2}\lambda_2(\Omega (\epsilon))=\lambda_2(\Omega
(0))-C'\epsilon^{(m-1)/2}(1+o(1)),
\end{equation*} for some $C'>0$
depending on $m$ and on $R$ only. Hence for $\epsilon$
sufficiently small $L(\epsilon)\Omega(\epsilon)$ is connected with
$\h(\partial(L(\epsilon)\Omega(\epsilon)))=\h(\partial
G_1)+\h(\partial G_2),$ and
$\lambda_2(L(\epsilon)\Omega(\epsilon))<\lambda_2(G_1\cup G_2)$.
This contradicts the hypothesis that $\Omega_{m,k}$ has two
components $G_1$ and $G_2$, whose union supports two eigenvalues.
\end{proof}

To prove Theorem \ref{The2}(iii) we note that by Lemma \ref{Lem3},
$\Omega_{m,2}$ is either connected or is the union of two
components supporting one eigenvalue each.  The latter is excluded
by Lemma \ref{Lem7}. So $\Omega_{m,2}$ is connected.

To complete the proof of Theorem \ref{The2} we let $k=3,4,\cdots$,
and $m=3,4,\cdots$. By Lemma \ref{Lem7} we may assume that $\o$
has at most one component supporting only one eigenvalue of $\o$.
So
\begin{equation*}
\o=\cup_{i=1}^{\omega_{m,k}} G_i,
\end{equation*}
where all components except possibly $G_1$ support at least two
eigenvalues. Hence $\omega_{m,k}\le \lfloor(k+1)/2\rfloor$. Let
$\h(\partial G_1)=a$. By Lemma \ref{Lem5} and Faber-Krahn we have
that \bee \ll_k \ge \l_1(G_1)\ge \l_1(B_m)\left( \frac{\h(\partial
B_m)}{a}\right)^{2/(m-1)}. \eee By Lemma \ref{Lem5} we also have
that for any $i \in \{2,3,\cdots\ ,\omega_{m,k} \},$ \bee \ll_k =
\max \{\l_j(G_i) : \l_j(G_i)\leq \ll_k\}. \eee By Krahn-Szeg\"o it
follows that for any $i \in \{2,3,\cdots\ ,\omega_{m,k} \},$
\begin{equation}\label{e34} \ll_k \geq 2^{2/m}
\l_1(B_m)\left( \frac{\h(\partial B_m)}{\h(\partial
G_i)}\right)^{2/(m-1)}, \end{equation}and in particular that
\begin{equation}\label{e35} \ll_k \geq 2^{2/m} \l_1(B_m)\left(
\frac{\h(\partial B_m)}{\min_{i\in \{2,\cdots,\omega_{m,k}\}}
\h(\partial G_i)}\right)^{2/(m-1)}.
\end{equation}
We have by \eqref{e22a} that
\begin{equation*}
\sum_{i=2}^{\omega_{m,k}}\h(\partial G_i)=1-a,
\end{equation*}
and so
\begin{equation*}
\min_{i\in \{2,\cdots,\omega_{m,k}\}} \h(\partial G_i) \leq
\frac{1-a}{\omega_{m,k}-1}.
\end{equation*}
Hence by \eqref{e35}
\begin{equation}\label{e38} \ll_k \geq
2^{2/m} \l_1(B_m) (\omega_{m,k}-1)^{2/(m-1)} \left(
\frac{\h(\partial B_m)}{1-a}\right)^{2/(m-1)}.
\end{equation}
Combining \eqref{e34} with \eqref{e38} yields
\begin{align*}
\ll_k
 &\geq \l_1(B_m) (\h(\partial B_m))^{2/(m-1)}\nonumber \\
&\times\max\{a^{-2/(m-1)},
2^{2/m}(\omega_{m,k}-1)^{2/(m-1)}(1-a)^{-2/(m-1)}\}.
\end{align*}
The right hand side of the inequality above attains its lower
bound for \bee a = ( 1+ (\omega_{m,k}-1)2^{(m-1)/m})^{-1}. \eee
Hence
\begin{equation}\label{e39} \ll_k \geq
  \l_1(B_m) (\h(\partial B_m))^{2/(m-1)}
 (1+(\omega_{m,k}-1) 2^{(m-1)/m})^{2/(m-1)}.
\end{equation}
 On the
other hand
\begin{equation}\label{e32b}
\ll_k \leq \l_k(B_m) (\h(\partial B_m))^{2/(m-1)}.
\end{equation}
Putting \eqref{e39} and \eqref{e32b} together gives that
\begin{equation*}
\lambda_k(B_m)\ge \lambda_1(B_m) (1+(\omega_{m,k}-1)
2^{(m-1)/m})^{2/(m-1)}.
\end{equation*}
This completes the upper bound in \eqref{e10}.

Next note that by \eqref{e41} and \eqref{e41a} we have that for
$k\le m+1$

\begin{align}\label{e42}
\omega_{m,k}&\le 1+\lfloor
2^{-(m-1)/m}((\lambda_2(B_m)/\lambda_1(B_m))^{(m-1)/2}-1)\rfloor
\nonumber \\ &=1+\lfloor
2^{-(m-1)/m}((j_{m/2}/j_{(m-2)/2})^{m-1}-1)\rfloor.
\end{align}
Numerical evaluation of the right hand side of \eqref{e42} for $3
\le m < 2^{15}$ using \cite{M} gives the upper bound for $\omega$
as advertised. For $m \ge 2^{15}$ we use that \cite{LL}
\begin{equation}\label{e43}
j_{\nu} = \nu + f \nu ^{1/3}+f_{\nu}\ \nu^{-1/3},\ 1\le \nu <
\infty,
\end{equation}
where $f=1.8557\cdots$ can be expressed in terms of the first
positive zero of an Airy function, and $0.500<f_{\nu}<1.537$.
Hence
\begin{equation}\label{e44}
j_{m/2}\le m/2 + f(m/2)^{1/3}+2(m/2)^{-1/3},
\end{equation}
and
\begin{equation}\label{e45}
j_{(m-2)/2}\ge (m-2)/2+ f((m-2)/2)^{1/3}.
\end{equation}
Combining \eqref{e44} and \eqref{e45} gives that for $m\ge 2^{15}$
\begin{equation}\label{e46}
\left(\frac{j_{m/2}}{j_{(m-2)/2}}\right)^2 \le e^{2+6m^{-1/3}}\le
e^{35/16}.
\end{equation}
So for $m\ge 2^{15}$ and $k \le m+1$
\begin{equation*}
\omega_{m,k} \le 1+\lfloor 2^{-1+2^{-15}}(e^{35/16}-1)\rfloor =4,
\end{equation*}
which completes the proof of Theorem \ref{The2}.

\section{Appendix}\label{sec5} Let $u:\Omega \mapsto \R$ be the unique weak solution
of
\begin{equation}\label{a16}
-\Delta_{\Omega} u=1
\end{equation}
with $u=0$ on  $\partial \Omega$. The torsional rigidity of
$\Omega$ is defined by
\begin{equation}\label{a17}
P(\Omega)=\int_{\Omega}u.
\end{equation}
$P$ is well defined since $u\ge0$. It is well known that
$P(\Omega)$ may be finite even if $|\Omega|=+\infty$. For example
if $\Omega$ is any open set in $\R^m$ for which
$-\Delta_{\Omega}\ge c_{\Omega} \delta^{-2}$ in the sense of
quadratic forms, and $\delta \in L^2(\Omega)$, where $\delta$ is
the distance to the boundary then $(2m)^{-1}\int_{\Omega}\delta^2
\le P(\Omega)\le c_{\Omega}^{-1}\int_{\Omega}\delta^2$
\cite{vdBG}.

Below we show that finite torsional rigidity implies discrete
spectrum of the Dirichlet Laplacian. In particular we obtain a
lower bound for $\lambda_k(\Omega)$ in terms of $k$ and
$P(\Omega)$. This lower bound does not satisfy Weyl asymptotics
for the reason explained above.
\begin{lemma}\label{app}
If $P(\Omega)<\infty$ then the spectrum of $-\Delta_{\Omega}$ is
discrete, and
\begin{equation}\label{a18}
\lambda_k(\Omega)\ge c(m)P(\Omega)^{-2/(m+2)}k^{2/(m+2)},
\end{equation}
where
\begin{equation}\label{a18a}
c(m)=(m+2)^{-1}(4\pi)^{m/(m+2)}(2\Gamma((2+m)/2))^{2/(m+2)}.
\end{equation}
\end{lemma}
\begin{proof}
Let $p_{\Omega}(x,y;t), x\in\Omega, y \in \Omega, t>0$ denote the
Dirichlet heat kernel for $\Omega$. It is well known that the
Dirichlet heat kernel is non-negative, monotone increasing in
$\Omega$, and that it satisfies the semigroup property. Moreover
\begin{equation*}
u(x)=\int_0^{\infty}dt\int_{\Omega} dy\ p_{\Omega}(x,y;t).
\end{equation*}
Let $0<\alpha<1$. By Tonelli's Theorem
\begin{align}\label{a20}
P(\Omega)&=\int_0^{\infty}dt\iint_{\Omega\times\Omega}dx dy\
p_{\Omega}(x,y;t)\nonumber \\
&=(1-\alpha )\int_0^{\infty}dt\iint_{\Omega\times\Omega}dx dy\
p_{\Omega}(x,y;(1-\alpha)t).
\end{align}
On the other hand by domain monotonicity
\begin{equation}\label{a21}
p_{\Omega}(x,y;\alpha t)\le p_{\R^m}(x,y;\alpha t)\le (4\pi\alpha
t)^{-m/2}.
\end{equation}
By \eqref{a20}, \eqref{a21} and the semigroup property
\begin{align}\label{a22}
P(\Omega)&\ge (1-\alpha )\int_0^{\infty}dt(4\pi \alpha
t)^{m/2}\iint_{\Omega\times \Omega} dx dy\
p_{\Omega}(x,y;(1-\alpha)t)p_{\Omega}(x,y;\alpha t)\nonumber \\
&=(1-\alpha)\int_0^{\infty}dt (4\pi \alpha
t)^{m/2}\int_{\Omega}dx\ p_{\Omega}(x,x;t).
\end{align}
Hence the heat semigroup is trace class, and
\begin{equation}\label{a23}
\int_{\Omega}dx\
p_{\Omega}(x,x;t)=\sum_{j=1}^{\infty}e^{-t\lambda_j(\Omega)}<\infty,\
t>0.
\end{equation}
By \eqref{a22} and \eqref{a23}
\begin{align}\label{a24}
P(\Omega)&\ge (1-\alpha)(4\pi
\alpha)^{m/2}\Gamma((2+m)/2)\sum_{j=1}^{\infty}\lambda_j(\Omega)^{-(2+m)/2}\nonumber
\\ & \ge(1-\alpha)(4\pi
\alpha)^{m/2}\Gamma((2+m)/2)k\lambda_k(\Omega)^{-(2+m)/2}.
\end{align}
Choosing $\alpha =m/(m+2)$ in \eqref{a24} gives \eqref{a18} with
\eqref{a18a}.
\end{proof}
\noindent\textbf{Acknowledgments}\ We wish to thank Giuseppe
Buttazzo and Thomas Hoffmann-Ostenhof for helpful discussions.

\end{document}